\documentclass [11pt]{amsart}
\usepackage {amsmath, amssymb, amscd, mathrsfs, url,pinlabel,hyperref, verbatim}
\usepackage[text={5.5in,9in},centering,letterpaper,dvips]{geometry}
\usepackage{color,multirow,dcpic,latexsym,pictexwd,graphicx,epstopdf}

\newtheorem {theorem}{Theorem}
\newtheorem {lemma}[theorem]{Lemma}

\newtheorem {corollary}[theorem]{Corollary}

\newtheorem {definition}[theorem]{Definition}

\newtheorem {question}[theorem]{Question}

\theoremstyle{remark}

\newtheorem {remark}[theorem]{Remark}

\numberwithin{equation}{section}
\numberwithin{theorem}{section}

\newcommand{\comments}[1]{}
\newcommand{\spinc}{$\mathrm{Spin}^c\;$}

\title{Contact Structures on AR-singularity links}
\author{\c{C}a\u gr{\i} Karakurt }
\address{Department of Mathematics, Bo\u{g}azi\c{c}i University, Bebek 34342}
\email{cagri.karakurt@boun.edu.tr}
\author{Fer\.{\i}t \"Ozt\"urk}
\address{Department of Mathematics, Bo\u{g}azi\c{c}i University, Bebek 34342}
\email{ferit.ozturk@boun.edu.tr}
\date{}

\begin{document}

\begin{abstract}

An isolated complex surface singularity induces a canonical contact structure on its link. In this paper, we initiate the study of the existence problem of Stein cobordisms between these contact structures depending on the properties of singularities. As a first step we construct an explicit Stein cobordism from any contact 3-manifold to the canonical contact structure of a proper almost rational singularity introduced by N\'emethi. We also show that the construction cannot always work in the reverse direction: in fact the U-filtration depth of contact Ozsv\'ath-Szab\'o invariant obstructs the existence of a Stein cobordism from a proper almost rational singularity to a rational one. Along the way, we detect the contact Ozsv\'ath-Szab\'o invariants of those contact structures fillable by a AR plumbing graph, generalizing an earlier work of the first author.
 
\end{abstract}
\maketitle

\section{Introduction}

Let $Z$ denote a complex analytic surface in $\mathbb{C}^N$ which has an isolated singularity at the origin. By intersecting $Z$ with a small sphere centered at 0, we get a closed oriented $3$-manifold $M$, which is called the link of the singularity. The distribution of complex tangencies on $M$ is a contact distribution. In general the $3$-manifold $M$ may be the link of many other analytically distinct isolated singularities, but the induced contact structures are known to be contactomorphic \cite{CNP}. If a $3$-manifold is realized as a link of an isolated singularity, then the associated contact structure is called the canonical contact structure on $M$.  Properties of these contact structures have been extensively studied in the literature  \cite{AkhO, BO,LO}.
One remarkable feature of  canonical contact structures is that they are all Stein fillable;  a deformation of the minimal resolution of a normal surface singularity and the Milnor fiber of a smoothable singularity determine Stein fillings of the canonical contact structure.

The purpose of the present paper is to address the existence problem of Stein cobordisms between canonical contact structures on the link manifolds of various classes of singularities. This is also related to the problem of symplectically embedding one Milnor fiber into another. As a first step we work on almost rational (AR) singularities, which have been introduced by A.~N\'emethi as an extension of rational singularities. Recall that a complex surface singularity is rational if its geometric genus is zero. M.~Artin found out that this is  equivalent to the case when any positive divisorial 2-cycle in a resolution of the singularity has nonpositive arithmetic genus; besides this last condition is independent of the resolution \cite{Art}. N\'{e}methi investigated the behavior of the arithmetic genus function (which is equal to $1-\chi$ where $\chi$ is as in \eqref{eqn:chi}) on the lattice of homology 2-cycles of a resolution of a normal surface singularity. Through his observation, he deduced that if a normal surface singularity is rational then  its link manifold $M$ is an L-space
\cite[Theorem~6.3]{N}; i.e. $M$ is a rational homology sphere and its Heegaard Floer homology is isomorphic  to that of a lens space. The converse of this statement was also  proved recently in \cite{N4}.

Now, an AR-singularity is one which admits a good resolution whose dual graph is a negative definite, connected tree and has the following property: by reducing the weight on a vertex 
we get the dual graph of a rational singularity. This property allows one to compute the Heegaard Floer homology of the link by a combinatorial process similar to Laufer's method of finding Artin's fundamental cycle \cite{N}. Even though the class of AR-singularities is restrictive, it is still large enough to contain the rational singularities and  all those singularities whose links are Seifert fibered rational homology spheres with negative definite plumbing graphs.  We call an almost rational singularity \emph{proper AR} if it is not rational.  After recalling some preliminaries on contact structures in Section~\ref{s:conpre} and on plumbings in Section~\ref{plumbing}, we present our first result in Section~\ref{construction} which claims the existence of a Stein cobordism from an arbitrary contact structure to the canonical contact structure of a proper AR-singularity.

\begin{theorem}\label{thm:steinAR}
Every closed, contact 3-manifold is Stein cobordant to the canonical contact structure of a proper AR-singularity.
\end{theorem}

The proof of  this theorem is an explicit construction of a Stein cobordism from the given contact manifolds to Brieskorn spheres.

By contrast there are obstructions to existence of Stein cobordisms going in the reverse direction. These obstructions come from Heegaard-Floer theory, particularly its plus flavor $HF^+$ \cite{OS4, OS5}. Recall that for any connected, closed, and oriented  $3$-manifold $M$ the Heegaard-Floer homology  group $HF^+(M)$ is a graded $\mathbb{F}[U]$-module, where $\mathbb{F}=\mathbb{Z}/2\mathbb{Z}$. Any oriented cobordism between two such manifolds induces a graded $\mathbb{F}[U]$ module homomorphism between the corresponding Heegaard-Floer homology groups. 

To any co-oriented contact structure $\xi$ on a $3$-manifold $M$, there is an associated element $c^+(\xi)\in HF^+(-M)$, which is a contactomorphism invariant of the contact structure and 
is natural under Stein cobordisms \cite{OS6}. In particular if $\xi$ is Stein fillable then $c^+(\xi)$ does not vanish since the tight structure on $S^3$ has non-vanishing  $c^+$. Moreover, $U(c^+(\xi))=0$. Utilizing this property, one can generate a numerical invariant of contact structures by letting
$$ \sigma(\xi)=-\mathrm{Sup} \left \{d\in \mathbb{N}\cup\{0\}:c^+(\xi)\in U^d \cdot HF^+(-M) \right \}.$$
The first author showed that $\sigma$ is monotone under Stein cobordisms and  can take all the values in the set $\{0,-1,-2,\dots,-\infty \}$ \cite{K}.   The computation of $\sigma(\xi)$ is in general hard, nevertheless we are able to show that it is zero for the canonical contact structures of proper AR-singularities. The Stein cobordism obstructions are obtained as immediate corollaries of this observation.

\begin{theorem}
\label{nocob}
Let $(M,\xi)$ be the canonical contact link manifold of a proper AR-singularity. Then  $\sigma(\xi)=0$. Hence there is no Stein cobordism from $\xi$ to 
\begin{enumerate}
\item any contact structure supported by a planar open book,
\item any contact structure on the link of a rational singularity, or
\item any contact structure with vanishing Ozsv\'ath-Szab\'o invariant.
\end{enumerate}
\end{theorem}
In the course of proving this theorem we give an explicit method to detect the contact invariant $c^+(\xi)$ in the homology of  graded roots of A.~N\'emethi, if $\xi$ satisfies a certain compatibility condition with a plumbing graph.  

 To detect $c^+(\xi)$ in the homology of  graded roots, one should understand the isomorphism 
\begin{equation} 
\Phi:HF^+(-M(\Gamma),\mathfrak{t}_{\mathrm{can}})\to \mathbb{H}(R_{\tau}).
\label{Fi}
\end{equation} 
Here the left side is the Heegaard Floer homology of the 3-manifold described via an $AR$-plumbing graph $\Gamma$ in the canonical \spinc structure $\mathfrak{t}_{\mathrm{can}}$. The right side is the homology of a graded root $R_{\tau}$, which is much easier to compute. In Section~\ref{HF} we describe the isomorphism $\Phi$ explicitly along with the required algebraic objects, tell how $c^+$ is detected in graded roots and finally prove Theorem~\ref{nocob}. In Section~\ref{examples} we present explicit examples.

\section{Contact Preliminaries}\label{s:conpre}

In this section we mention some background material about contact geometry. Our main purpose is to set up our terminology. For a thorough discussion see \cite{OzbStip,Gei, CE}.

Contact structures on 3-manifolds can be studied topologically via open books through the Giroux corespondence.
We say that a contact structure $\xi$ on $M$ is compatible with an open book in $M$ if 
on the oriented binding of the open book a contact form of $\xi$ is a positive volume form  and
away from the binding  $\xi$ can be isotoped through contact structures to be arbitrarily close  to the tangent planes of the pages of the open book. 
Giroux correspondence states  that this compatibility relation is in fact a one-to-one correspondence between contact 3-manifolds up to contact isotopy and open books up to positive stabilizations.  

Suppose $(W,J)$ is a compact complex surface with oriented boundary $-M_1\cup M_2$ that admits a strictly plurisubharmonic Morse function $\phi:W\to[t_1,t_2]$ such that $M_i=\phi^{-1}(t_i)$, for $i=1,2$, and $-dJ^*d\phi$ is a symplectic form. Then the set of complex tangencies on $M_i$ constitute a contact structure $\xi_i$, for $i=1,2$. In this case we say that $W$ is a Stein cobordism from $(M_1,\xi_1)$ to  $(M_2,\xi_2)$. If $M_1=\emptyset$ we say that $W$ is a Stein filling for  $(M_2,\xi_2)$.

Establishing the existence of a Stein cobordism between given contact manifolds is a delicate problem. Etnyre and Honda proved that an overtwisted contact 3-manifold is Stein cobordant to any contact $3$-manifold and that for any contact $3$-manifold $M$ there is a Stein fillable one to which $M$ is  Stein cobordant  \cite{EH}. The obstructions usually utilize Floer type theories. The first author used Heegaard Floer homology to prove non-existence of Stein cobordisms between certain contact manifolds \cite{K}; see also Sections \ref{HF} and \ref{examples} below. Another powerful tool  which obstructs more generally exact symplectic cobordisms is Latschev and Wendl's algebraic torsion which was  built in the framework of symplectic field theory \cite{LW}. In the appendix of the same paper Hutchings described the same obstruction in  embedded contact homology. Translation of the latter in the Heegaard Floer setting was recently given by Kutluhan et al \cite{KMVW}.   

Starting with a contact $3$-manifold $(M_1,\xi_1)$ we can build a Stein cobordism as follows. First take the product $M_1\times [0,1]$ and equip it with the standard Stein structure. Then attach $1$-handles to the upper boundary.  Next attach $2$-handles along Legendrian knots in $M_1\sharp_n S^1\times S^2$ with framing one less than the contact framing.  By the topological characterization of Stein cobordisms  given by Eliashberg \cite{Eli}, see also \cite[Theorem~1.3]{Go}, the standard Stein structure on $M_1\times [0,1]$ extends uniquely  over such handles and all Stein cobordisms can be constructed this way. We are going to need a well-known partial translation of this recipe in terms of open books.

\begin{lemma}\label{l:monostei}
There exists a Stein cobordism from $(M_1,\xi_1)$ to $(M_2,\xi_2)$ if  there exist open books $(S_i,\phi_i)$ compatible with $(M_i,\xi_i)$ for $i=1,2$, such that $S_1$ and $S_2$ are homeomorphic surfaces and $\phi_2$ is obtained by multiplying  $\phi_1$  with right handed Dehn twists along non-separating curves.  
\end{lemma}

\begin{proof}
By Legendrian realization principle \cite[Theorem~3.7]{H}, each non-separating curve on a page can be isotoped to a Legendrian curve staying on the same page such that the page framing and the contact framing agree. Then adding right handed Dehn twists to the monodromy along these curves precisely corresponds  to adding a Stein handle along the Legendrian curves. 
\end{proof}
\section{Preliminaries on Plumbings}
\label{plumbing}
Let $\Gamma$ be a weighted connected tree with vertices $\{b_j: j \in \mathcal{J}\}$ for some finite index set $\mathcal{J}$. Let $e_j$ denote the weight of the $j$th vertex  $b_j$ for all $j\in \mathcal{J}$. We construct a $4$-manifold $X(\Gamma)$ as follows: For each vertex $b_j$,  take a disk bundle over a sphere whose Euler number is $e_j$ and plumb two of these together whenever there is an edge connecting the vertices. We denote the boundary 3-manifold by $M(\Gamma)$. We shall call such a graph  a plumbing graph. These graphs naturally arise in  singularity theory as good dual resolution graphs where   vertices correspond to  exceptional divisors  and the edges represent their intersection. See \cite{N3} for a survey on this topic. 

\subsection{The lattices $L$ and $L'$} The 4-manifold $X(\Gamma)$ admits a handle decomposition without $1$-handles, so the second homology group  $L=H_2(X(\Gamma);\mathbb{Z})$ is freely generated by the the fundamental classes of the zero-sections of the disk bundles in the plumbing construction. Hence we have a generator of $L$ for each vertex of $\Gamma$. By abuse of notation we denote the generator in $L$ corresponding  to the vertex $b_j$ by the same symbol. The intersection form $(\cdot,\cdot)$ on $L$ is a symmetric bilinear function naturally characterized by $\Gamma$.

It is known that a plumbing graph can be realized as a good dual resolution graph of a singularity if and only if the intersection form is negative definite. Henceforth, we shall assume that all the plumbing graphs we consider satisfy this property. Therefore $L$ is a lattice and we have the short exact sequence
\begin{equation}\label{eq:dia1}
\begin{CD}
0 @>>> L  @>PD>>L' @>>> H @>>>0 
\end{CD}
\end{equation}
where $L'$ is the dual lattice $\mathrm{Hom}_{\mathbb{Z}}(L,\mathbb{Z})\cong H^2(X(\Gamma);\mathbb{Z})\cong H_2(X(\Gamma),M(\Gamma);\mathbb{Z}) $  with $PD(x)=(x,\cdot)$, and $H=H_1(M(\Gamma);\mathbb{Z})$ .

We say that  $k\in  L'$ is \emph{characteristic} if for every vertex $b_j$ of $\Gamma$, 
we have $k(b_j)+e_j\equiv 0 \text{ mod } 2$. Let $\mathrm{Char}(\Gamma)$ denote the set of characteristic elements in $L'$. The lattice $L$ naturally acts on $\mathrm{Char}(\Gamma)$ by the rule $x\ast k =k +2PD(x)$ for every $x\in L$.

The characteristic cohomology class $K\in L'$ 
satisfying $K(b_j)=-e_j-2$ for all $j\in \mathcal{J}$ is called the canonical class. 
For each $k\in \mathrm{Char}(\Gamma)$, we define the function $\chi_k:L\to\mathbb{Z}$ by
\begin{equation}
\label{eqn:chi}
\chi_k(x)=-(k(x)+(x,x))/2.
\end{equation}
We simply use the symbol $\chi$ to denote $\chi_K$. 

\subsection{\spinc structures}\label{s:spinc}
Every element of $\mathrm{Char}(\Gamma)$ uniquely defines a \spinc structure on $X(\Gamma)$. Two such \spinc structures induce the same \spinc structure on $M(\Gamma)$ if and only if the corresponding characteristic cohomology classes are in the same $L$-orbit.
For a fixed \spinc structure $\mathfrak{t}$ of $M(\Gamma)$, let $\mathrm{Char}(\Gamma,\mathfrak{t})$ denote the set of all characteristic cohomology classes which restrict to $\mathfrak{t}$ on $M(\Gamma)$.  Since $M(\Gamma)$ is a rational homology sphere,  it has finitely many \spinc structures; these are in one-to-one correspondence with elements of $H_1(M(\Gamma),\mathbb{Z})$. We can express $\mathrm{Char}(\Gamma)$ as a disjoint union of finitely many $\mathrm{Char}(\Gamma,\mathfrak{t})$s. In this paper, the default  \spinc structure, denoted by $\mathfrak{t}_{\mathrm{can}}$, is the one induced by the canonical class $K$.

\subsection{Almost Rational Graphs} 
A plumbing graph $\Gamma$ is said to be \emph{rational}  if $\chi(x)\geq 1$ for any $x>0$.  Call a vertex $b_j$ in a plumbing graph a bad vertex if    $-e_j$ is greater than the valency of the vertex.  It is known that graphs with no bad vertices are rational. Moreover, rational graphs are closed under decreasing weights and taking subgraphs. 
 
A graph is said to be \emph{almost rational} (or AR for short) if by decreasing the weight of a vertex $b_0$, we get a rational graph. Note that such a distinguished vertex $b_0$ need not be unique. Graphs with at most one bad vertex are AR \cite[Section~8]{N}. In particular  plumbing graphs of Seifert fibered rational homology spheres are AR. A plumbing graph is said to be proper almost rational if it is almost rational but not rational. 

We say that a surface singularity is rational (respectively, AR and proper AR) if it admits a good resolution graph which is rational (respectively AR and proper AR). For example, for pairwise relatively prime integers $p,q$ and $r$, the link of the Brieskorn singularity
\begin{equation}\label{eq:Brieskorn}
x^p+y^q+z^r=0,
\end{equation}
is a Seifert fibered integral homology sphere which is  called  Brieskorn sphere $\Sigma(p,q,r)$. Such a singularity is proper AR unless $(p,q,r)=(2,3,5)$ \cite{R2,E,N}.

\section{Construction of Stein Cobordisms}
\label{construction}
The purpose of this section is to prove Theorem \ref{thm:steinAR}. In fact we will show that the proper AR-singularity in the  theorem can be chosen as a Brieskorn singularity.  For arbitrary positive integers $g$ and $n$, consider the triple $(p,q,r)$ with 
$$p=2,\quad q=2g+1,\text{ and } r=(4g+2)n+1.$$
\noindent Note that $p,q,$ and $r$ are pairwise relatively prime for all $g$ and $n$, so the corresponding Brieskorn sphere $\Sigma(p,q,r)$ is an integral homology sphere, and the corresponding Brieskorn singularity is proper AR. 

We shall describe an abstract open book which supports the canonical contact structure on $\Sigma(p,q,r)$. Let $S_g$ denote a compact orientable surface of genus $g$ with one boundary component. Let $\phi_{g,n}$ denote the diffeomorphism on $S_g$ which is a product of right handed Dehn twists along simple closed curves given by
\begin{equation}\label{eq:monodromy}
\phi_{g,n}=(t_{a_1}t_{a_2}\dots t_{a_{2g}})^{(4g+2)n+1},
\end{equation}
\noindent where the curves $a_1,a_2,\dots,a_{2g}$ form a chain as in Figure \ref{fig:Page_Chain}
\begin{figure}[h]
	\includegraphics[width=0.60\textwidth]{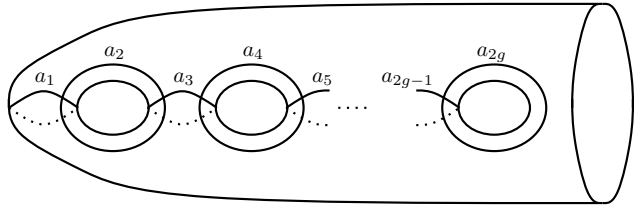}
	\caption{Curves $a_1,\dots a_{2_g}$ which form a chain on $S_g$.}
	\label{fig:Page_Chain}
\end{figure}

The chain relation \cite[Proposition 4.12]{FM} tells that 
\begin{equation}\label{eq:chain}
(t_{a_1}t_{a_2}\dots t_{a_{2g}})^{(4g+2)}=t_\delta 
\end{equation}
\noindent where $\delta$ is a simple closed curve on $S_g$ which is parallel to the boundary component. Since $\delta$ does not intersect with any of $a_1,\dots,a_{2g}$, the Dehn twist $t_\delta$ commutes with all of $t_{a_1},\dots,t_{a_{2g}}$.  Hence $\phi_{g,n}$ can also be written as 
$$\phi_{g,n}=(t_{a_1}t_{a_2}\dots t_{a_{2g}})t_\delta ^n $$   
\begin{lemma}\label{lem:sup}
The canonical contact structure on $\Sigma (p,q,r)$ is supported by $(S_g,\phi_{g,n})$. 
\end{lemma}

\begin{proof} Consider the Milnor fiber $M(p,q,r)$ of the singularity \eqref{eq:Brieskorn}. By definition we have $\partial M(p,q,r)=\Sigma (p,q,r)$. 
We shall construct a holomorphic Lefschetz fibration on $M(p,q,r)$ whose fibers are diffeomorphic to $S_g$ and monodromy factorizes exactly as in \eqref{eq:monodromy}. Our argument will be based on a work of Loi and Piergallini \cite{LP}, which relates branched covers of Stein 4-manifolds to Lefschetz fibrations. 

Since $p=2$, the Milnor fiber $M(p,q,r)$ is a 2-fold branched cover of $B^4=\{ (y,z)\in\mathbb{C}^2\,:\, |y|^2+|z|^2\leq \epsilon \}$, branched along the Milnor fiber $M(q,r)$ of the plane curve singularity
\begin{equation}\label{eq:curve}
y^q+z^r=0
\end{equation}
\noindent Let $h:M(p,q,r)\to B^4$ be the covering map. It is well known that the link of the plane curve singularity \eqref{eq:curve} is the $(q,r)$-torus knot $T(q,r)$ in $S^3=\partial B^4$, and the Milnor fiber $M(q,r)$ is diffeomorphic to the minimal genus Seifert surface of $T(q,r)$ in $S^3$. Identify $B^4 \cong B^2\times B^2$  
using the complex coordinates $(y,z)$ and consider the projection map to the second factor $\pi_2:B^4\to B^2$. The restriction of this map to $M(q,r)$ is a 
simple $q$-fold branched covering 
whose singular points, called the twist points, are in one-to-one correspondence with crossings of $T(q,r)$ represented as a braid with $q$ strands and $r$ twists. Here it is important that all the crossings of $T(q,r)$ are positive otherwise charts describing simple branched cover around the twist points would not be compatible with the complex orientation on $M(q,r)$. Let $s_1,s_2,\dots,s_{qr}$ be the twist points on $M(q,r)$. Without loss of generality we may assume that they are mapped to distinct points  $z_1,z_2,\dots,z_{r}$ under $\pi_2$. By \cite[Proposition 1]{LP}, 
the composition $f:=\pi_2\circ  h$ is a holomorphic Lefschetz fibration whose set of singular values  is $\{z_1,z_2,\dots,z_{r}\}$.

Next we determine the regular fibers of the Lefschetz fibration $f:M(p,q,r)\to B^2$. Away from the twist points, any  disk $B^2\times \{ \mathrm{pt} \}$ in $B^4$ intersects $M(q,r)$ at $q$ points.  Since $q=2g+1$, each  such disk lifts under $h$ to a genus $g$ surface $F$   with one boundary component; this forms a regular fiber of $f$. The restriction map $h|_F:F\to B^2\times \{\mathrm{pt}\}$ 
is modeled by taking the quotient of $F$ by the hyperelliptic involution whose $2g+1$ fixed points map to $B^2\times \{ \mathrm{pt} \} \cap M(q,r)$. See the marked points on $F$ in Figure \ref{fig:braid_lift}. 

When a disk $B^2\times \{\mathrm{pt} \}$ intersects $M(q,r)$ at a twist point, its lift under $h$ is a singular fiber of $f$. We can describe the vanishing cycle of each singular fiber using the corresponding crossing of $T(q,r)$.  Any arc $\gamma$ on $B^2\times \{ \mathrm{pt}\}$
 connecting two strands of the braid $T(q,r)$ lifts to a unique  simple closed curve $\alpha$ on the regular fiber $F$. In Figure \ref{fig:braid_lift}, we indicated these arcs and curves using different colors. If a crossing exchanges these two strands connected by $\gamma$  then the corresponding singular fiber has vanishing cycle $\alpha$. Each such vanishing cycle contributes a right handed Dehn twist along $\alpha$  if the corresponding crossing is positive
(see \cite[Proposition 1]{LP} or \cite[Lemma 4.2]{BE}).
Following the braid direction we see that the monodromy is as described in \eqref{eq:monodromy}.
\begin{figure}[h]
	\includegraphics[width=0.50\textwidth]{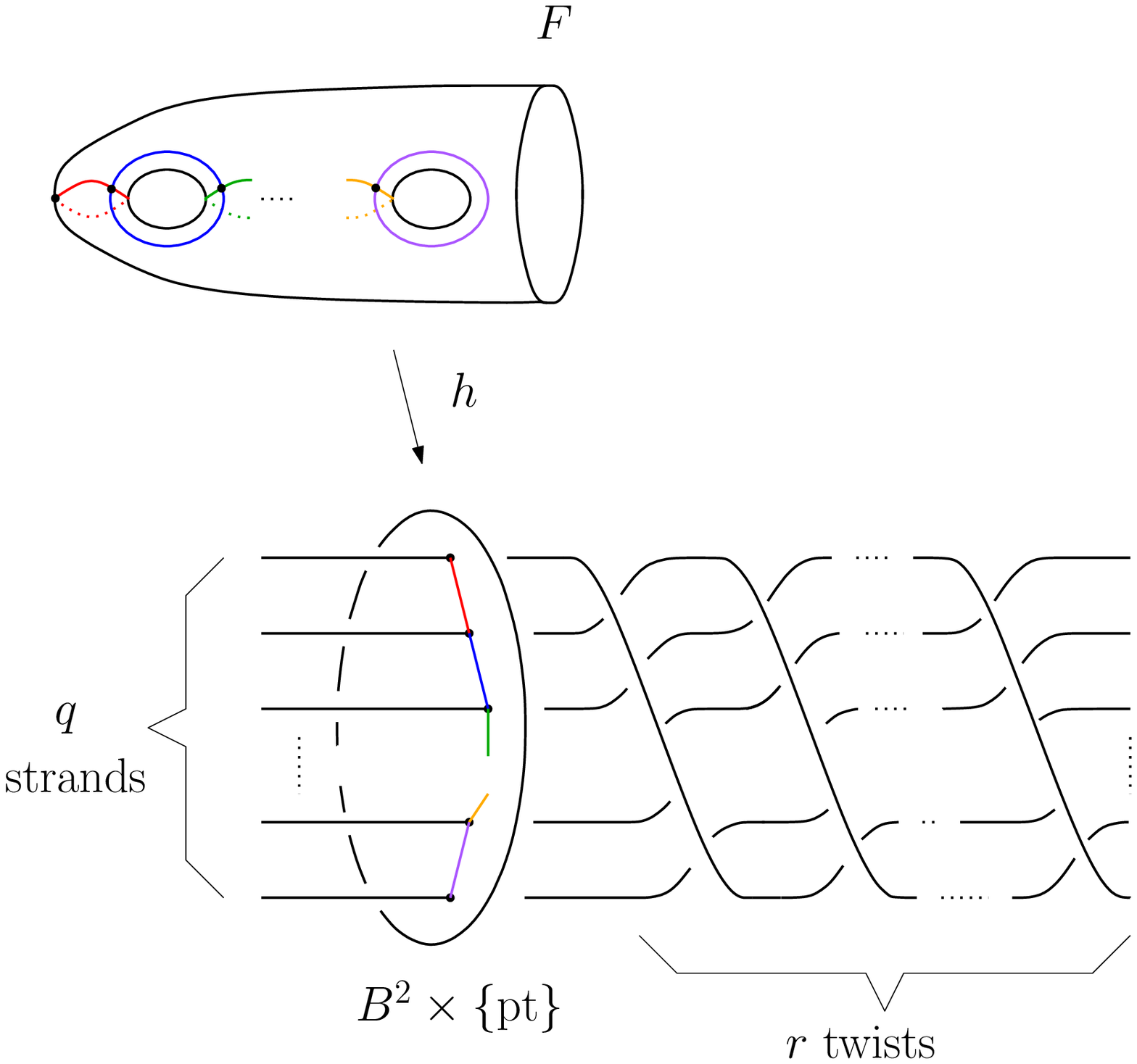}
	\caption{.}
	\label{fig:braid_lift}
\end{figure}

Having shown that the fibers of $f$ are diffeomorphic to $S_g$ and the total monodromy of $f$ agrees with $\phi_{g,n}$, we conclude that the restriction of $f$ on $\partial M(p,q,r)$ is the open book $(S_g,\phi_{g,n})$. Since the fibers of $f$ are complex submanifolds of $M(p,q,r)$, the open book $(S_g,\phi_{g,n})$ supports the canonical contact structure. 
\end{proof}

\begin{remark}
An alternative proof of the above lemma goes as follows: Using handlebody techniques and the fact that $\Sigma(p,q,pqn+1)$ is $-1/n$ surgery on $(p,q)$-torus knot, one can directly verify that the total space of the open book $(S_g,\phi_{g,n})$ is $\Sigma(2,2g+1,(2g+1)n+1)$. 
Using the chain relation and a result of W.~D.~Neumann and A.~Pichon \cite[Theorem 2.1]{NP}, one  can show that the open book supports the canonical contact structure.
\end{remark}

\begin{lemma}\label{lem:cob}
Given any contact 3-manifold $(M,\xi)$, there exist $g,n_0\in \mathbb{N}$ such that for every $n\geq n_0$ there is  a Stein cobordism from $(M,\xi)$ to the canonical contact structure on the Brieskorn sphere $\Sigma(2,2g+1,(2g+1)n+1)$.
\end{lemma}

\begin{proof}
Take an open book supporting $(M,\xi)$. If the pages of this open book have more than one boundary component, we can positively stabilize the open book to reduce the number of boundary components to one at the expense of increasing the page genus. Let $g$ denote the genus of the pages of the resulting open book after the stabilizations. Fix an identification of the pages with $S_g$, let $\phi$ denote the monodromy.  

Write the monodromy as a product of Dehn twists about non-separating curves $c_1,c_2,\dots,c_k$ in $S_g$,
\begin{equation}\label{eq:factor}
\phi=t_{c_1}^{\epsilon_1}t_{c_2}^{\epsilon_2}\cdots t_{c_k}^{\epsilon_k}, \text{ where } \epsilon_i\in \{-1,+1\} \text{ for all } i=1,\dots,k
\end{equation}
\noindent The exponents in the above equation emphasize that the factorization $\phi$ can contain both right handed and left handed Dehn twists. We shall make the monodromy agree with $\phi_{g,n_0}$ for some large $n_0\in \mathbb{N}$ by multiplying it by only right handed Dehn twists. If $\epsilon_1=-1$ we simply multiply $\phi$ by $t_{c_1}$ from left to cancel the first term. If $\epsilon_1=+1$, we need to do more work! Since $c_1$ is non-separating, a self-diffeomorphism $\psi$ of $S_g$ sends $a_{2g}$ to $c_1$ 
(see Figure~\ref{fig:Page_Chain}).  
Multiply both sides of the  chain relation \eqref{eq:chain} by $\psi^{-1}$ from right and by $\psi$ from left, and use the fact that $t_\delta$ is in the center of the mapping class group of $S_g$, to  see that
$$(t_{\psi(a_1)}t_{\psi(a_2)}\dots t_{\psi(a_{2g})})^{(4g+2)}=t_\delta.$$ 
Hence using  $(4g+2)(2g)-1$ right handed Dehn twists we can trade $t_{c_1}$ with $t_\delta$. The latter can be put at the end of the factorization \eqref{eq:factor}. Applying the same recipe to the remaining $c_i$'s we get the monodromy $t_\delta^{n_0}$, where $n_0$ is the number of positive exponents appearing in the factorization \eqref{eq:factor}. By adding Dehn twists appearing in the left hand side of the chain relation \eqref{eq:chain} as many times as necessary, we get the monodromy agree with $\phi_{g,n}$ for any $n\geq n_0$.  

During the process we made only two kinds of modifications on the open book: positive stabilizations and adding a right handed Dehn twist to the monodromy. By Lemma~\ref{l:monostei} the required Stein cobordism exists.
\end{proof}
 
\begin{proof}[Proof of Theorem \ref{thm:steinAR}]
Immediately follows from Lemma \ref{lem:sup} and Lemma \ref{lem:cob}.
\end{proof}

\section{Heegaard Floer homology, graded roots and Stein cobordism obstructions }
\label{HF}

To understand the image of $c^+(\xi)$  under $\Phi$ in (\ref{Fi}), we analyze the isomorphism $\Phi$  carefully by considering its factorization as follows:

\begin{equation}\label{eq:dia}
\begin{CD}
HF^+(-M(\Gamma)) @>\Phi_1>> \mathbb{H}^+(\Gamma)  @>\widetilde{\Phi}>> (\mathbb{K}^+(\Gamma))^* @>\Phi_2>> \mathbb{H}(R_\tau) 
\end{CD}
\end{equation}

In order to introduce the notation and recall the definitions, we recite briefly  the construction of $HF^+$, $\mathbb{H}^+$, $\mathbb{K}^+$, and the isomorphisms $\Phi_1$ and  $\widetilde{\Phi}$ in Sections~\ref{HF+}-\ref{K+}
and the graded roots, their homology and the isomorphism  $\Phi_2$  in Sections~\ref{rootR}-\ref{laufer}. We describe how to detect root vertices of graded roots in  $(K^+(\Gamma))^*$  in Section~\ref{detecting} and distinguish the contact invariant in $(K^+(\Gamma))^*$ in Section~\ref{c+}. Finally we prove Theorem~\ref{nocob} in Section~\ref{pf2}.

\subsection{Plus flavor of Heegaard Floer homology}
\label{HF+}
Heegaard Floer theory is a package of prominent invariants for $3$-manifolds, \cite{OS4, OS5} see also \cite{J,M} for recent surveys. To every closed and oriented $3$-manifold $M$, and any \spinc structure $\mathfrak{t}$ on $M$, one associates an $\mathbb{F}:=\mathbb{Z}/2\mathbb{Z}$-vector space $HF^+(M,\mathfrak{t})$ which admits a canonical $\mathbb{Z}/2\mathbb{Z}$ grading.  In the case that the first Chern class of the \spinc structure is torsion (in particular when $M$ is a rational homology sphere), the $\mathbb{Z}/2\mathbb{Z}$-grading lifts to an absolute $\mathbb{Q}$-grading. There is also an endomorphism, denoted by  $U$, decreasing the degree by $2$  so as to make $HF^+(M,\mathfrak{t})$  an $\mathbb{F}[U]$-module. For example when the $3$-manifold is the 3-sphere and $\mathfrak{t}_0$ is its unique \spinc structure, we have $HF^+(S^3,\mathfrak{t}_0)=\mathcal{T}^+_{0}$ where $\mathcal{T}^+$ 
stands for the $\mathbb{F}[U]$-module $\mathbb{F}[U,U^{-1}]/U\mathbb{F}[U]$ 
and $\mathcal{T}^+_{d}$ denotes the one in which 
the lowest degree element is supported in degree $d$. 
More generally when $M$ is a rational homology sphere, it is known that the Heegaard Floer homology is given by $HF^+(M,\mathfrak{t})=\mathcal{T}^+_{d}\oplus HF_{\mathrm{red}}(M,\mathfrak{t})$ where $HF_{\mathrm{red}}(M,\mathfrak{t})$ is a finitely generated $\mathbb{F}$-vector space (and hence a finitely generated $\mathbb{F}[U]$-module) \cite{OS2}. In the sequel, we shall assume that all our $3$-manifolds are rational $3$-spheres.

Heegaard Floer homology groups behave functorially under cobordisms in the following sense: Suppose 
$M_1$ and $M_2$ are connected, closed and oriented $3$-manifolds
and there is a cobordism $W$ from $M_1$ to $M_2$; i.e. the oriented boundary of $W$ is 
$\partial W=(-M_1)\cup M_2$. For every \spinc structure $\mathfrak{s}$ on $W$, let $\mathfrak{t}_1$ and $\mathfrak{t}_2$ denote the induced \spinc structures on $M_1$ and $M_2$ respectively. Then we have an $\mathbb{F}[U]$-module homomorphism  
$$ F_{W,\mathfrak{s}}:HF^+(-M_2,\mathfrak{t}_2)\to HF^+(-M_1,\mathfrak{t}_1).$$

\subsection{The graded module   $\mathbb{H}^+(\Gamma)$}  
\label{H+} The  differential in the chain complex defining Heegaard Floer homology counts certain types of holomorphic disks. This aspect makes the computation of Heegaard Floer homology groups difficult in general. On   the other hand if a three manifold arises from a plumbing construction, there is a purely algebraic description for its Heegaard Floer homology.

 Let $\Gamma$ be a negative definite plumbing graph. 
For every \spinc structure $\mathfrak{t}$ of $M(\Gamma)$, let $\mathbb{H}^+(\Gamma,\mathfrak{t})$ be the subset of the set of functions Hom$(\mathrm{Char}_{\mathfrak{t}}(\Gamma),\mathcal{T}^+)$  satisfying the following property. For every $k\in \mathrm{Char}_\mathfrak{t}(\Gamma)$ and every $j\in\mathcal{J}$, let $n$ be the integer defined by
\begin{equation}\label{e:kay}\chi_{k,j}=\chi_k(b_j)=-(k(b_j)+e_j)/2.\end{equation} 
Then for every positive integer $m$, we require
$$U^{m-\chi_{k,j}}\phi(k+2PD(b_j))=U^m\phi(k) \text{ whenever } \chi_{k,j}\leq0, \text{ and}$$
$$U^{m}\phi(k+2PD(b_j))=U^{m+\chi_{k,j}}\phi(k) \text{ whenever } \chi_{k,j}>0.$$
We define  $\mathbb{H}^+(\Gamma)=\bigoplus_{\mathfrak{t}}\mathbb{H}^+(\Gamma,\mathfrak{t})$, which has readily an $\mathbb{F}[U]$-module structure. Moreover the  conditions
above provide the existence of a suitable grading on it. 
A map $\Phi_1: HF^+(-M(\Gamma),\mathfrak{t})\to \mathbb{H}^+(\Gamma,\mathfrak{t} )$ can be described as follows: Remove a disk from the $4$-manifold $X(\Gamma)$ and regard it as a cobordism $\widetilde{X}$ from $-M(\Gamma)$ to $S^3$. For each characteristic cohomology class $k\in \mathrm{Char}_{\mathfrak{t}}(\Gamma)$ we have a map $F_{\widetilde{X},k}:HF^+(-M(\Gamma),\mathfrak{t})\to HF^+(S^3)=\mathcal{T}^+_{0}$. For any Heegaard Floer homology class $c\in  HF^+(-M(\Gamma))$ and for any characteristic cohomology class $k$, we define $\Phi_1(c)(k):=F_{\widetilde{X},k}(c)$. Thanks to adjunction relations, this map is a well defined homomorphism which is in fact an isomorphism when $\Gamma$ is AR. This was shown by   Ozsv\'ath and Szab\'o for the special case where $\Gamma$ has at most one bad vertex \cite{OS1} and later by N\'emethi in general \cite{N}.

\subsection{The dual of $\mathbb{H}^+(\Gamma)$ } 
\label{K+}
There is a simple description of the dual of $\mathbb{H}^+(\Gamma)$. We use the notation $U^m\otimes k$ to denote a typical element of $\mathbb{Z}^{\geq 0}\times\mathrm{Char}(\Gamma, \mathfrak{t})$.  The elements of the form $U^0\otimes k$ are simply indicated by $k$.  Define an equivalence relation $\sim$ on $\mathbb{Z}^{\geq 0}\times\mathrm{Char}(\Gamma,\mathfrak{t})$ by the following rule. Let $j\in\mathcal{J}$ 
and $n$ be as in (\ref{e:kay}). Then we require
$$U^{m+n}\otimes(k+2PD(b_j))\sim U^m\otimes k \text{ if }n\geq 0,$$
\noindent and 
$$ U^{m}\otimes(k+2PD(b_j))\sim U^{m-n}\otimes k \text{ if }n< 0.$$ 
Let $\mathbb{K}^+(\Gamma,\mathfrak{t})$ denote the set  of equivalence classes. Let $(\mathbb{K}^+(\Gamma,\mathfrak{t}))^*$ denote its dual.  For any non-negative integer $l$, let Ker~$U^{l+1}$ denote the subgroup of  
$\mathbb{H}^+(\Gamma,\mathfrak{t})$ which is the kernel of the multiplication by $U^{l+1}$. 
The map 
$$\widetilde\Phi_l: \mathrm{Ker}\,(U^{l+1})\to\mathrm{Hom} \left (\frac{ \mathbb{K}^+(\Gamma,\mathfrak{t})}{\mathbb{Z}^{\geq l}\times{\mathrm{Char}_{\mathfrak{t}}(\Gamma)}},\mathbb{F} \right ),$$
given by the rule
$$\widetilde{\Phi}_l(\phi)(U^m\otimes k)=(U^m\phi(k))_0 $$
\noindent is an isomorphism for every $l$ \cite[Lemma~2.3]{OS1}. Here
$(\cdot)_0$ denotes the projection to the degree 0 subspace of  $\mathcal{T}^+$. Since every element of $\mathbb{H}^+(\Gamma, \mathfrak{t})$ lies in some $ \mathrm{Ker}\,(U^{l+1})$, the maps $\widetilde\Phi_l$ give rise to an isomorphism $\widetilde{\Phi}:\mathbb{H}^+ (\Gamma,\mathfrak{t})\to (\mathbb{K}^+(\Gamma,\mathfrak{t}))^*$.

\subsection{Graded roots} 
\label{rootR}
Let $R$ be an infinite tree. Denote its  vertex set and edge set  by $\mathcal{V}(R)$ and $\mathcal{E}(R)$ respectively. Let $\chi:\mathcal{V}(R)\to \mathbb{Z}$ satisfy the following properties.

\begin{enumerate}
	\item $\chi(u)-\chi(v)=\pm 1$, if $[u,v]\in \mathcal{E}(R)$.
	\item $\chi(u)>\mathrm{min}\{\chi (v),\chi (w)\}$, if $[u,v]\in \mathcal{E}(R)$, and $[u,w]\in \mathcal{E}(R)$.
	\item $\chi$ is bounded below.
	\item $\chi^{-1}(n)$ is a finite set for every $n$.
	\item $|\chi^{-1}(n)|=1$ for $n$ large enough.
\end{enumerate}
Such a pair $(R,\chi)$ is called a \emph{graded root}. When the grading function $\chi$ is apparent in the discussion we simply drop it from our notation and use $R$ to denote a graded root.

Next we describe a particular graded root produced by a function $\tau:\mathbb{Z}^{\geq 0}\to \mathbb{Z}$ that is non-decreasing after a finite index. For each $i \in \mathbb{Z}^{\geqslant 0}$ consider the graded tree $R_i$ with vertices $\{v_i^m\}_{m\geq \tau(i)}$ and the edges $\{[v_i^m,v_i^{m+1}]\}_{m\geq \tau(i)}$ with grading $\chi(v_i^m)=m$. 
Define an equivalence relation on the disjoint union $\coprod_i R_i$ of trees as follows: 
$v_i^m \asymp v_j^n$ and $[v_i^m,v_i^{m+1}]\asymp [v_j^n,v_j^{n+1}]$ if and only if $m=n$ and 
$m\geq \tau(l)$ for all $l$ between $i$ and $j$. 
Then $R_{\tau}=\coprod_i R_i / \asymp$ is a tree with vertices the equivalence classes $\overline{v_i^m}$ and with the induced  grading $\chi(\overline{v_i^m}) = m$. 

To each graded root $(R,\chi)$ as above, with vertex set $\mathcal{V}$ and edge set $\mathcal{E}$, one can associate a graded $\mathbb{F}[U]$ module $\mathbb{H}(R,\chi)$ as follows.
As a set, $\mathbb{H}(R,\chi)$ is the set of functions $\phi: \mathcal{V}\to \mathcal{T}^+_{0}$ satisfying 
\begin{equation}
\label{deg:Hr}
U\cdot\phi (u)=\phi (v),
\end{equation}
\noindent whenever $[u,v]\in \mathcal{E}$ with $\chi(u)<\chi (v)$. The $U$-action on $\mathbb{H}(R,\chi)$ is defined by the rule $(U\cdot \phi) (v)=U(\phi(v))$. The grading on $\mathbb{H}(R,\chi)$ is defined in the following way. An element $\phi \in \mathbb{H}(R,\chi)$ is homogeneous of degree $d$ if for every $v\in\mathcal{V}$, $\phi(v)$ is homogeneous of degree $d-2 \chi(v)$.

\subsection{Laufer sequences} 
\label{laufer} To simplify our discussion we will work with the canonical \spinc structure from now on; some of the discussion below works for a general \spinc structure, though. 
Our aim is to describe the isomorphism 
$$\Phi_2:(\mathbb{K}^+ (\Gamma, \mathfrak{t}_\mathrm{can}))^*\to \mathbb{H}(R_\tau), $$
discovered by N\'emethi, in a slightly different setup which suits to our needs; what we do
below is nothing but rephrasing the findings in \cite{N}.
Now, the isomorphism $\Phi_2$ is induced by  a bijection  between the dual objects:
$$\Phi_3:\mathcal{V}(R_\tau) \to  \mathbb{K}^+(\Gamma,\mathfrak{t}_\mathrm{can}).$$
Here $R_\tau$ is a graded root associated to a $\tau$ function and $\mathcal{V}(R_\tau)$ is its vertex set. We now describe the construction of the  function $\tau$.

Recursively form a sequence $(k(i))_{i=0}^\infty$ in $\mathrm{Char}(\Gamma, \mathfrak{t}_{\mathrm{can}})$ as follows: start with the canonical class $k(0)=K$. Suppose $k(i)$ has  already been constructed. We find $k(i+1)$ using the algorithm below.
\begin{enumerate}
\item \label{alg:N1} We construct a computational sequence $z_0,z_1,\dotsm z_l$. Let $z_0=k(i)+2\mathrm{PD}(b_0)$. Suppose $z_m$ has been found.  If there exists $j\in \mathcal{J}-\{0\}$ such that 
$$z_m(b_j)=-e_j$$
then we let $z_{m+1}= z_{m}+2\mathrm{PD(b_j)}$. 
\item If there is no such $j$, stop.  Set $l=m$ and $k(i+1)= z_{l}$. 
\end{enumerate}  
The sequence $(k(i))_{i=0}^\infty$ is called the Laufer sequence of the AR graph $\Gamma$ associated with the canonical \spinc structure. This sequence depends on the choice of the distinguished vertex $b_0$ but   is independent from our choice of vertices in step \ref{alg:N1} of the above algorithm.  Define $\chi_{i,0}=\chi_{k(i)}(b_0)$. Note that  since the elements of each  computational sequence described above satisfy $z_m\sim z_{m+1}$ for every $m=0,\dots,l-1$, the vectors $k(i)$ satisfy the following relations in $\mathbb{K}^+(\Gamma,\mathfrak{t}_{\mathrm{can}})$:
\begin{align*}
U^{\chi_{i,0}}\otimes k(i) \sim  k(i+1) & \text{ if }\chi_{i,0}\geq 0,  \text{ and }\\
k(i)\sim U^{-\chi_{i,0}}\otimes k(i+1) & \text{ if }\chi_{i,0}< 0.
\end{align*}
Let  $\tau(n)=\sum_{i=0}^{n-1}\chi_{i,0}$, with $\tau(0)=0$.  It can be shown that there exists an index $N$ such that  $\tau(i+1)\geq \tau(i)$ for all $i\geq N$ \cite[Theorem~9.3(a)]{N}. 
Hence $\tau$ defines a graded root $R_\tau$ and indices beyond $N$ do not give essential information about  $R_\tau$. Moreover, for the canonical Spin$^c$ structure, 
$\tau(N)\geq 2$ and $\tau(i)\leq 1$ for all $i\leq N-1$ \cite[Theorem~6.1(d)]{N}.
In that case one can stop the computation procedure of the Laufer sequence once $\tau\geq 2$. Now with a little effort, one can observe in that the map $\Phi_3:\mathcal{V}(R_\tau)\to \mathbb{K}^+(\Gamma)$ is defined by $\Phi_3(\overline{v_i^m})=U^{m-\tau(i)}\otimes k(i)$
(after the proofs of \cite[Theorem~9.3(b)]{N} and \cite[Proposition~4.7]{N}). One can check that $\Phi_3$ is well defined and injective and moreover $\Phi_3$ is in fact a bijection so it induces an $\mathbb{F}[U]$-module isomorphism $\Phi_2:(\mathbb{K}^+(\Gamma))^*\to \mathbb{H}(R_{\tau})$, which shifts grading by  $(K^2+|\mathcal{J}|)/4$ 

\begin{remark}
The Laufer sequence $\{x(i)\}$ in \cite{N} resides in $L$ while $\{k(i)\}$ here resides in $\mathrm{Char}(\Gamma, \mathfrak{t}_{\mathrm{can}})$. These two Laufer sequences  are related by $$ k(i)=K+2PD(x(i)).$$
\end{remark}

\begin{remark}
The sequence $(\tau(i))_{i=0}^\infty$ contains a lot of redundant elements. The finite subsequence $(\tau(n_i))$ consisting of local maximum and local minimum values of $\tau$ is sufficient to construct the graded root.
\end{remark}

\begin{remark}
An algorithm similar to what we have described above can be utilized to compute the Heegaard Floer homology groups for an arbitrary \spinc structure $\mathfrak{t}$. The only new necessary input is the  distinguished representative of  the \spinc structure inside $\mathrm{Char}(\Gamma,\mathfrak{t})$. Interested reader can consult \cite[Section 5]{N}.  
\end{remark}

\subsection{Detecting root vertices} 
\label{detecting} Ozsv\'ath and Szab\'o used a variation of the above algorithm to determine $\mathrm{Ker}U$ in $(\mathbb{K}^+(\Gamma))^*$ \cite[Section 3.1]{OS1}. Their elements  are also visible in the Laufer sequence. In a graded root $R$, say that a vertex is a \emph{root vertex} if it has valency $1$. The following lemma identifies root vertices of $R_\tau$ with elements of $\mathrm{Ker}U$.
\begin{lemma}\label{l:ktog}
 Given $k\in  \mathbb{K}^+(\Gamma,\mathfrak{t}_{\mathrm{can}})$  such that $k^* \in \mathrm{Ker}U$, there exists an element $k(i_0)$ of the Laufer sequence such that $k\sim k(i_0)$. This element is unique in the following sense: if $k(i_0)\sim k(i_1) \sim k$  and $i_0<i_1$ then $\tau(i)=\tau(i_0)$ for all $i$ satisfying $i_0\leq i\leq i_1$.  As a result $\Phi_3^{-1}(k)$ is the root vertex of the branch in the graded root $R_\tau$ corresponding to $\tau(i_0)$. 
\end{lemma}

\begin{proof}
Since $\Phi_3$ is surjective, $k\sim U^{n}\otimes k(i_0)$ for some $i_0,n \in \mathbb{Z}^{\geq 0}$. Since $k^*$ is in $\mathrm{Ker}U$, $k$ does not admit any representation of the form $U^n\otimes k'$ unless $n=0$. Then the vertex $\overline{v_{i_0}^{\tau(i_0)}}$ in the graded root $R_\tau$  must have valency $1$ since otherwise  $v_{i_0}^{\tau(i_0)} \asymp  v_{i}^{\tau(i)+m}$ for some $i,m \in \mathbb{Z}^{\geq 0}$, implying that $k\sim U^m\otimes k(i)$, a case which we have dismissed. The same argument  proves that $\tau$ must be constant between $i_0$ and $i_1$ if $k(i_0)\sim k(i_1) \sim k$.
\end{proof}

\begin{lemma} \label{l:laufdetect}
If $k \in \mathrm{Char}(\Gamma,\mathfrak{t}_{\mathrm{can}})$ satisfies 
$$e_j+2 \leq k(b_j)\leq -e_j-2,\text{ for all } j\in \mathcal{J},$$
\noindent then $k^* \in \mathrm{Ker}U$ and there exists a unique $i_0\in \mathbb{Z}^{\geq 0}$ such that $k=k(i_0)$. Consequently $\Phi_3^{-1}(k)$ is the root vertex of the branch in the graded root $R_\tau$ corresponding to $\tau(i_0)$. Moreover,  the index $i_0$ is the component of the vector $\mathrm{PD}^{-1}(k(i_0)-K)/2$ on $b_0$.
\end{lemma}

\begin{proof}
By \cite[Proposition~3.2]{OS1}, $k$ forms a good full path of length $1$ so it does not admit any representation of the form $U^n\otimes k'$ with $n\geq 1$, and it is the only element in its $\sim$ equivalence class. By Lemma~\ref{l:ktog}, we  must have $k= k(i_0)$ for some element $k(i_0)$ of the Laufer sequence. The claim about the index is in fact satisfied by every element of the Laufer sequence \cite[Lemma 7.6 (a)]{N}.
\end{proof}

\subsection{Contact invariant} 
\label{c+} To any co-oriented contact structure $\xi$ on a $3$-manifold $M$, one associates an element $c^+(\xi)\in HF^+(-M)$. This element is an invariant of the contact structure and it satisfies the following properties:

\begin{enumerate}
\item $c^+(\xi)$ lies in the summand $HF^+(-M,\mathfrak{t}_\xi)$ where $\mathfrak{t}_\xi$ is the \spinc structure uniquely determined  by the homotopy class of $\xi$.
\item $c^+(\xi)$ is homogeneous of  degree $-d_3(\xi)-1/2$, where $d_3(\xi)$ is the $3$-dimensional invariant of $\xi$ of Gompf \cite{Go}.
\item When $\xi$ is overtwisted, $c^+(\xi)=0$.
\item When $\xi$ is Stein fillable $c^+(\xi)\neq 0$.
\item We have $U(c^+(\xi))=0$.
\item $c^+(\xi)$ is natural under Stein cobordisms.
\end{enumerate}

Our aim is to understand where the contact invariant falls under the isomorphism described in \eqref{eq:dia}. This could be difficult for a general contact structure. We need a certain type of compatibility of the contact structure with the plumbing. 

\begin{definition}  
Let $\Gamma$ be a plumbing graph. Suppose $\xi$ is a contact structure on $M(\Gamma)$. We say that $\xi$ is compatible with $\Gamma$ if  the following are satisfied:
\begin{enumerate}
\item The contact structure $\xi$ admits a  Stein filling whose total space, possibly after finitely many blow-ups, is $X(\Gamma)$. 
\item The induced \spinc structure agrees with that of the canonical class on $M(\Gamma)$. (This condition is automatically satisfied when $M(\Gamma)$ is an integral homology sphere.)
\end{enumerate}
\end{definition}
Note that  canonical contact structures of singularities are compatible with the dual resolution graphs. 
Furthermore if $\xi$ is compatible with $\Gamma$ then $X(\Gamma)$ is a strong symplectic filling for $\xi$.
We shall denote by $J$ and $c_1(J)$, the corresponding almost complex structure on $X(\Gamma)$  and its first Chern class respectively. 
\begin{theorem}\label{cagri}(\cite[Proposition~1.2]{K}) 
Let $\xi$ be compatible with $\Gamma$. Then we have $\widetilde{\Phi}\circ \Phi_1(c^+(\xi))=(c_1(J))^*$.
\end{theorem}

\begin{proof}
If $X(\Gamma)$ has no $-1$ spheres,  the total space of the Stein filling is $X(\Gamma)$ itself. Then result is an easy consequence of the definitions of the maps $\Phi_1$ and $\widetilde{\Phi}$ and  Plamenevskaya's theorem \cite{P} which says that,
\begin{equation}\label{eq:pla} F_{\widetilde{X},k}(c^+(\xi))=\left \{
\begin{tabular}{lr} 
1& \text{ if } $k \sim c_1(J)$,\\
0 & \text{ if } $k\not \sim c_1(J)$.
\end{tabular}\right . 
\end{equation}

In the case that $X(\Gamma)$ contains $-1$ spheres, we blow them down until we get a Stein filling of $\xi$. Applying Plamenevskaya's theorem there and using blow-up formulas we see that \eqref{eq:pla} still holds. 
\end{proof}

\begin{corollary}\label{cor:can}
 We have $\widetilde{\Phi}\circ \Phi_1(c^+(\xi_\mathrm{can}))=K^*$ where $K$ is the canonical class.
\end{corollary}

\subsection{Proof of Theorem~\ref{nocob}}
\label{pf2}
By Corollary~\ref{cor:can},  $\widetilde{\Phi}\circ \Phi_1(c^+(\xi_\mathrm{can}))=K^*$ where $K$ is the canonical class which is also the first element $k(0)$ of the Laufer sequence. Under the correspondence described in Lemma~\ref{l:ktog}, this element is associated with the root vertex of the branch of $\tau(0)$.  We will be  done once we prove this branch has length one, which implies that $\Phi(c^+(\xi_{\mathrm{can}}))$ is not in the image of $U^n$ for any $n> 0$. Since $\Phi$ is an $\mathbb{F}[U]$ module isomorphism, the same must hold for $c^+(\xi_{\mathrm{can}})$.

By definition, $\tau(0)=0$, and a direct computation shows $\tau(1)=1$. Now, by \cite[Theorem~6.1(d)]{N} 
it follows that $\#\tau^{-1}(m)=1$ whenever $m\geq 1$; equivalently if $\tau(n)>1$ for some $n$ 
then $\tau$ is increasing beyond $n$. Then we have two cases: either $\tau$ is always 
increasing or there is some $n$ for which $\tau(n)<1$ and $\tau(j)=1$ for $1\leq j \leq n$. In the former case we get $\mathbb{H}(R_\tau)=\mathcal{T}^+_0$. However this is equivalent to having $\Gamma$ rational (\cite[Theorem~6.3]{N}), which we have dismissed by assumption.
In the latter case, there are more than one root of the  tree $R_{\tau}$  which have non-positive degree. 
In particular the branch corresponding $\tau(0)$ has length $1$ (as illustrated in Figure~\ref{fig:tree}), and therefore we have $\sigma(\xi)=0$.

\begin{figure}[h]
	\includegraphics[width=0.40\textwidth]{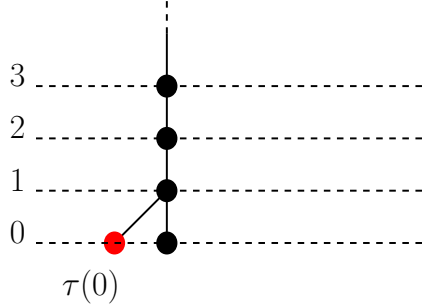}
	\caption{If $\Gamma$ is proper AR, this tree embeds into the graded root $R$ where the red vertex is identified with $\tau (0)$.}
	\label{fig:tree}
\end{figure}

Finally, there can be no Stein cobordism from $(M(\Gamma),\xi)$ to $(M',\eta)$ if $\sigma(\eta)<0$. This follows from the naturality of the invariant $c^+$ under Stein cobordisms and the fact that 
the invariant $\sigma$ increases under a Stein cobordism \cite[Theorem~1.5]{K}. To finish the proof
we recall that for each of the cases in the theorem, $\sigma(\eta)=-\infty$: 
this is immediate if $c^+(\xi)=0$;
for the case $\eta$ is a planar contact structure, this claim is just \cite[Theorem~1.2]{OSS}; 
for the case when $M'$ is the link of a rational singularity, this is a consequence of  \cite[Theorem~6.3]{N}.
\hfill $\Box$

\section{Examples}\label{examples}
The main result of our paper concerns canonical contact structures but with the techniques we developed in this paper we can in fact compute the $\sigma$ invariant of any contact structure which is compatible with an AR plumbing graph. Here we give a few examples. 

\subsection{Contact structures on $\Sigma(2,3,11)$} We start with a simple but an instructive example. The Brieskorn sphere $\Sigma(2,3,11)$ is the boundary of the graph in Figure~\ref{fig:plumbing} 
which we denote by $\Gamma$. We index the vertices $\{b_0,b_1,\ldots,b_8\}$ 
of $\Gamma$ so that
$b_0$ is the one with adjacency 3 and $b_8$ is the one with weight $-3$.
We know that $\Gamma$ is  proper AR since $(2,3,11)$ are pairwise relatively prime.

\begin{figure}[h]
	\includegraphics[width=0.40\textwidth]{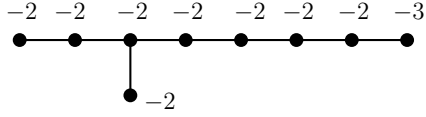}
	\caption{Plumbing graph of the Brieskorn sphere $\Sigma(2,3,11)$.}
	\label{fig:plumbing}
\end{figure}

One may find compatible Stein structures on the $4$-manifold $X(\Gamma)$ by choosing  the Legendrian
attaching circles of $2$-handles corresponding to vertices with the prescribed intersection matrix
such that the smooth framing is one less than the Thurston-Bennequin framing.  Note that there is a 
unique way of doing this for each $(-2)$-framed vertex, but the $(-3)$-framed $2$-handle can be 
Legendrian realized in two different ways. We fix an orientation on this vertex and distinguish these 
two cases according to their rotation numbers \cite[Theorem~1.2]{LM}. Hence $X(\Gamma)$ has two natural distinct Stein structures  whose Chern classes are $k_{\pm}=[0,\dots,0,\pm 1]$ where we write an element 
$k\in H^2(X(\Gamma),\mathbb{Z})$ in the form $[k(b_1),\dots,k(b_s)]$ in the dual basis. Denote the corresponding contact structures on boundary $\xi_{\pm }$.

\begin{figure}[h]
	\includegraphics[width=1.0\textwidth]{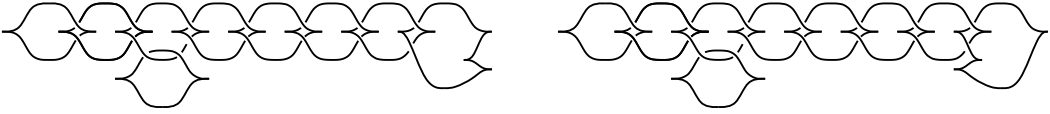}
	\caption{Stein handlebody pictures of $\xi_{\pm }$.}
	\label{fig:legplumbing}
\end{figure}

The Heegaard Floer homology of $\Sigma(2,3,11)$ is well known; we can readily compute it here
by considering the Laufer sequence $k(n)$ and the corresponding values $\tau(n)$. Moreover from Lemma~\ref{l:laufdetect},  we know that $k_\pm$ will appear in the Laufer sequence. 
The first several $k(n)$ and $\tau(n)$ are as follows: 
\begin{align*}
k_+= k(0)&=(0,0,0,0,0,0,0,0,1), \quad &\tau(0)=0 \\
k(1)&=(2,-2,0, -2, 0, 0, 0, 0, -3),\quad &\tau(1)=1 \\
k(2)&=(2, 0, -2, 0, 0, 0, 0, -2, -1), \quad &\tau(2)=1 \\
k(3)&=(2, -2, 0, 0, 0, 0, -2, 0, -1), \quad &\tau(3)=1 \\
k(4)&=( 2, 0, 0, -2, 0, -2, 0, 0, -1),\quad  &\tau(4)=1 \\
k(5)&=(4, -2, -2, 0, -2, 0, 0, 0, -1), \quad &\tau(5)=1 \\
k_-=k(6)&=(0, 0, 0, 0, 0, 0, 0, 0, -1), \quad &\tau(6)=0 
\end{align*}
Moreover it can be proven that $\tau(n+1)\geq \tau(n)$ for every $n\geq 6$. An indirect way to do this is by observing that $\tau(13)=2$ and $\tau(n)=1$ for $7\leq n <13$ and then  employing
\cite[Theorem~6.1(d)]{N} to conclude that $\tau$ is increasing for $n\geq 13$. Hence we obtain
the graded root $R_{\tau_K}$ as shown in Figure~\ref{fig:graded2311}. By the previous discussion, the  contact invariants $c^+(\xi_{\pm})$ correspond to the two roots of the tree. As a result, we have $\sigma(\xi_\pm)=0$. Notice that $k_+=K$, and so $c^+(\xi_+)=c^+(\xi_\mathrm{can})$. 

\begin{figure}[h]
	\includegraphics[width=0.60\textwidth]{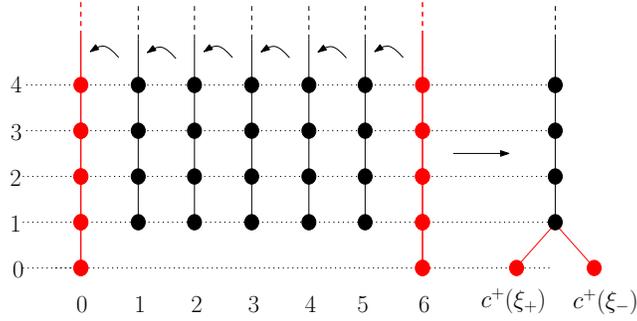}
	\caption{The graded root of $\Sigma(2,3,11)$ with the canonical \spinc structure. }
	\label{fig:graded2311}
\end{figure}
Finally we determine the Heegaard Floer homology by computing the homology of the graded root and shifting the degree by $(K^2+9)/4$. We get that $HF^+(-\Sigma(2,3,11))=\mathcal{T}^+_{(-2)}\oplus \mathbb{F}_{(-2)}$ and $c^+(\xi_\pm)$ are two distinct elements of degree $-2$ which project non-trivially to the reduced Floer homology.   

\subsection{Stein fillable contact structures of arbitrarily large $\sigma$} Consider the infinite family of Brieskorn spheres $M_n=\Sigma(3,3n+1, 9n+2)$, $n\geq 1$ given by plumbing graph $\Gamma_n$ in Figure~\ref{fig:plumpqpq-1}. For every $m=0,\dots,n-1$ we have the  Stein structure $J_m$ corresponding to the Legendrian handle attachments along Legendrian unknots corresponding to the vertices of $\Gamma_n$ with framing $tb-1$. In order to get the correct framing we have to stabilize $(-3)$-and $(-n-1)$-framed unknots once and $n-1$ times respectively. To get the Stein structure $J_m$  we do one left stabilization on $(-3)$-framed unknot, and  $m$ right  and  $n-m-1$ left stabilizations on $(-n-1)$-framed unknot see Figure~\ref{fig:legplumbing2}. Orienting each unknot clockwise we fix a basis for $b_0,\dots,b_{9n+5}$ for $L$ using the indices as show in Figure~\ref{fig:plumpqpq-1}. Then the number $c_1(J_m)(b_j)$ is given by the rotation number of the Legendrian unkot corresponding to the $j$th vertex.  Hence we have $c_1(J_m)=[0,1,0,0,n-2m-1,0,\dots,0]$. Clearly for every $m=0,\dots,n-1$, the contact structure $\xi_m$ induced by $J_m$ is compatible with $\Gamma_n$. Hence by Lemma~\ref{l:laufdetect}, each $c_1(J_m)$ appears in the Laufer sequence of $\Gamma_n$. In fact we have $c_1(J_m)=k(i_m)$, where $i_m=3m(9n+2)$.

\begin{figure}[h]
	\includegraphics[width=.50\textwidth]{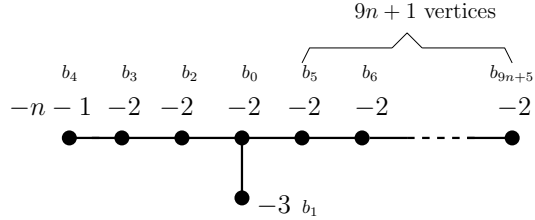}
	\caption{Plumbing graph  $\Gamma_n$.}
	\label{fig:plumpqpq-1}
\end{figure}

\begin{figure}[h]
	\includegraphics[width=0.70\textwidth]{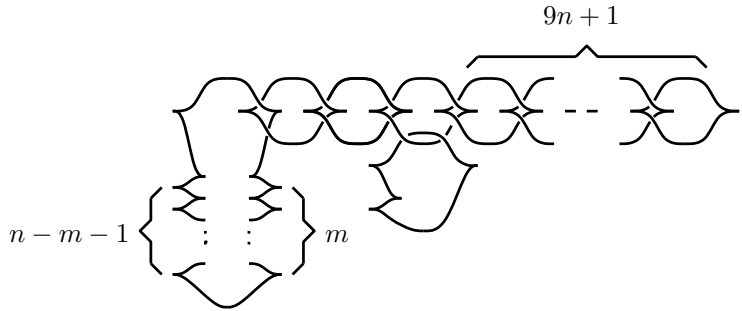}
	\caption{Stein handlebody diagram of $J_m$}
	\label{fig:legplumbing2}
\end{figure}

We need to decide the branch lengths of the corresponding root vertices in the graded root. Borodzik and N\'emethi worked out the combinatorics of the $\tau$ function. 
\begin{lemma}\cite[Propositon 4.2]{BN}  Suppose $p$ and $q$ are relatively prime positive integers. Let $\mathcal{S}_{p,q}$ be the semigroup of $\mathbb{N}$ generated by $p$ and $q$ including $0$. Let $\delta=(p-1)(q-1)/2$. Consider the function $\tau:\mathbb{Z}^{\geq 0}\to \mathbb{Z}$ associated to the Brieskorn sphere $\Sigma(p,q,pq-1)$.  

The function $\tau$ attains its local minima at $a_t=t(pq-1)$ for $0\leq t \leq 2\delta -2$, and local maxima at $A_t=tpq+1$ for $0\leq t\leq 2\delta -3$. Moreover for any $0\leq t\leq 2\delta -3$, one has
\begin{align*}
\tau(A_t)-\tau (a_{t+1})&=\#\{s\not \in \mathcal{S}_{p,q}\, : \, s\geq n+2 \}>0, \\
\tau(A_t)-\tau (a_{t})&=\#\{s\in \mathcal{S}_{p,q}\, : \, s\leq n \}>0. 
\end{align*} 

\end{lemma}

Applying the above lemma for the case $p=3$ and $q=3n+1$, we see that $\sigma(\xi_m)=-m$. Hence by choosing $n$ and $m$ appropriately we realize any negative integer as the $\sigma$ invariant of a contact structure. Of course $\xi_m$ cannot be isomorphic to the canonical contact structure unless $m=0$. Therefore the following problem is natural.

\begin{question}
Is it possible to realize any negative integer as the $\sigma$ invariant of the canonical contact structure of a singularity?
\end{question}
\section{Acknowledgments}
The first author  is supported by a TUBITAK grant BIDEB 2232 No: 115C005. The second author is grateful to the Institut Camille Jordan, Universit\'{e} Lyon 1, where part of this work was completed.
\bibliography{References}
\bibliographystyle{plain}

\end{document}